\documentclass[11pt]{amsart}
\usepackage{amsmath, amsthm, amscd, amssymb}

\newtheorem{thm}{Theorem}[section]
\newtheorem*{thm*}{Theorem}
\newtheorem{cor}[thm]{Corollary}
\newtheorem*{cor*}{Corollary}
\newtheorem{lem}[thm]{Lemma}
\newtheorem{prop}[thm]{Proposition}

\newtheorem*{con*}{Conjecture}
\newtheorem*{prob*}{Problem}

\theoremstyle{definition}

\theoremstyle{remark}
\newtheorem{rem}[thm]{Remark}

\newcommand{\A}{\mathcal{A}}
\newcommand{\B}{\mathcal{B}}
\newcommand{\C}{\mathcal{C}}

\newcommand{\h}{\mathcal{H}}

\newcommand{\M}{\mathcal{M}}

\newcommand{\F}{\mathcal{F}}

\newcommand{\V}{\mathcal{V}}

\newcommand{\bK}{\mathbf{K}}

\newcommand{\bbZ}{\mathbb{Z}}

\newcommand{\hol}{\text{Hol}}
\newcommand{\aut}{\text{Aut}}
\newcommand{\inn}{\text{Inn}}
\newcommand{\<}{\langle}
\renewcommand{\>}{\rangle}
\def\gl{{\rm GL}_2({\mathbb Z})}

\begin{document}
\title{On the $K$-theory of Certain Extensions of Free Groups}

\author{Vassilis Metaftsis}
\address{Department of Mathematics
University of the Aegean,
Karlovassi, Samos, 83200 Greece}
\email{vmet@aegean.gr}

\author[Stratos Prassidis]{Stratos Prassidis}
\address{Department of Mathematics
University of the Aegean,
Karlovassi, Samos, 83200 Greece}
\email{prasside@aegean.gr}

\maketitle

\section{Introduction}

The Fibered Farrell--Jones Conjecture is the main conjecture in geometric topology.
It is used for the calculation of the obstruction groups that appear in geometric rigidity and classification problems.

In this paper we are interested in the $K$-theory FJC and its variation $K$-FJCw, with finite wreath products.
The $K$-FJCw has been proved for an extensive list of classes of groups. One notable case which remains open is
the group $\aut(F_n)$, the automorphism group of the free group on $n$ letters. In \cite{mepr} the $K$-FJC is proved
for $n = 2$. Actually, in \cite{mepr}, the $K$-FJC is proved for $\hol(F_2)$, the holomorph of $F_2$. We notice that the
extension to $K$-FJCw is a direct computation.
In this paper we extend the result to certain subgroups of $\aut(F_n)$ that are constructed from $\hol(F_2)$. 
More precisely, there is a monomorphism $\hol(F_n) \to \aut(F_{n+1})$. This way we construct a sequence of
groups with:
$$\h_{(1)} = \hol(F_2), \; \h_{(n)} = F_{n+1} {\rtimes} \h_{n-1}, \; n \ge 2.$$
Notice that $\h_{(n)} < \hol(F_{n+1})$.
The main theorem of the paper is the following.

\begin{thm*}[Main Theorem] 
The $K$-FJCw holds for the groups $\h_{(n)}$.
\end{thm*}

As an application of the Main Theorem, we calculate the lower $K$-theory groups of $\h_{(n)}$:
\begin{enumerate}
\item $K_i(\bbZ\h_{(n)}) = 0$, $i \le -1$.
\item ${\tilde K}_0(\bbZ\h_{(n)}) \cong NK_0(\bbZ D_4) {\oplus} NK_0(\bbZ D_4)$.
\item $Wh(\h_{(n)}) \cong NK_1(\bbZ D_4) {\oplus} NK_1(\bbZ D_4)$.
\end{enumerate}

The main point of the general FJCw is that the $K$-theory of a group can be computed from the $K$-theory of its 
virtually cyclic subgroups. These are of three types: finite groups, groups that admit an epimorphism to $\bbZ$
with finite kernel, groups that admit an epimorphism to $D_{\infty}$ (the infinite dihedral group) with finite kernel. In 
\cite{dqr}, it was shown that the first two classes are enough for the $K$-FJCw. That means that the $K$-groups of a group
can be computed from the $K$-theory of finite subgroups and groups of the form $F{\rtimes}{\bbZ}$ with $F$ finite.
For the proof of the main theorem, we use two properties of the FJCw:
\begin{enumerate}
\item If the FJCw holds for a group $G$, it holds for all the subgroups of $G$.
\item Let 
$$1 \to H \xrightarrow{f} G \xrightarrow{g} K \to 1$$ be an exact sequence of groups. We assume that:
\begin{enumerate}
\item The FJCw holds for $K$.
\item The FJCw holds for $g^{-1}(V)$ where $V$ is a virtually cyclic subgroup of $K$.
\end{enumerate}
Then the FJCw holds for $G$
\end{enumerate}
In \cite{mepr}, it was shown that the finite subgroups of $\hol(F_2)$ are isomorphic to one of the groups of the following
list:
$${\bbZ}/2{\bbZ}, \; {\bbZ}/3{\bbZ}, \; {\bbZ}/4{\bbZ}, \; D_2, \; D_4.$$
The only subgroups of the second type are isomorphic to ${\bbZ}/2{\bbZ}{\times}{\bbZ}$. That implies that the finite 
subgroups of $\h_{(n)}$ are isomorphic to one of the above list. Also, we show that the subgroups of the second type are isomorphic to products of finite groups times $\bbZ$. In other words, semidirect products do not appear. 

For the proof of the main theorem, we use induction and Property (2) of the FJCw. We show that the groups that are
the inverses images of virtually cyclic groups of first or second type are either hyperbolic groups or CAT(0)-groups
for which the $K$-FJCw holds.

For the calculations of the lower $K$-groups, we notice that the part of $K$-theory of $\h_{(n)}$ that is detected 
from the finite groups vanishes. The result follows from the calculation of the cockerel of the map from the 
$K$-theory detected from the finite subgroups to the total $K$-theory. For this, we use \cite{ba}.

%%%%%%%%%%%%%%%%%%%%%%%%%%%%%%%%%
\section{Preliminaries and Notation}\label{sec-not}
%%%%%%%%%%%%%%%%%%%%%%%%%%%%%%%%%

For a group $G$ let $\aut(G)$ the group of automorphisms of $G$ and $\hol(G)$ the universal split extension determined by $G$:
$$1 \to G \to \hol(G) \to \aut(G) \to 1.$$
In general, there is an embedding $E: \hol(G) \to \aut(G*\bbZ)$ given by: for $g\in G$,
$$E(g)(x) = \left\{
\begin{array}{cr}
x, & \text{for}\; x\in G\\
gxg^{-1}, & \text{for}\; x\in \bbZ
\end{array}\right.$$
and for ${\alpha} \in \aut{G}$,
$$E({\alpha})(x) = \left\{
\begin{array}{cr}
{\alpha}(x), & \text{for}\; x\in G\\
x, & \text{for}\; x\in \bbZ
\end{array}\right.$$
Thus, we can define the split group extension $(G * \bbZ) {\rtimes} Å(\hol(G)) < \hol(G*\bbZ) < \aut((G*\bbZ)*\bbZ)$. 

Inductively, we define ${\h}_{(i)}(G)$ to be:
$${\h}_{(0)}(G) = G, \; {\h}_{(1)}(G) = \hol(G), \; {\h}_{(n)}(G) = (G * F_{n-1}) {\rtimes} Å({\h}_{(n-1)}(G)), \;\; n \ge 2.$$ where $F_{n+1}$ is the group on $(n+1)$-generators.
We write ${\h}_{(n)} =  {\h}_{(n)}(F_2)$, for the group corresponding to $F_2$. Then there is a split exact sequence:
$$1 \to F_{n+1} \to {\h}_{(n)} \to E({\h}_{(n-1)}) \to 1.$$

We are interested in the Fibered Farrell-Jones Conjecture (FJC) for the groups $\h_{(n)}$. We will review 
the general constructions. Let $G$ be a group and $\C$ be a class of subgroups. Then $E_{\C}G$
denotes the classifying space of the class $\C$. We are interested in the following classes of subgroups of $G$:
\begin{itemize}
\item $1$, the class of the trivial subgroup.
\item $\F$, the class of finite subgroups.
\item $\F\B\C$, the class of finite-by-cyclic subgroups.
\item $\V \C$, the class of virtually cyclic subgroups.
\item $\A ll$, the class of all subgroups.
\end{itemize}
It is obvious that $1 \subset \F \subset \F\B\C \subset \V \C \subset \A ll$. Instead of the classical $K$-theoretic FJC, we will
consider the $K$-theoretic Isomorphic Conjecture with coefficients in an additive category $\A$.
It is known that this implies also the Fibered Isomorphism Conjecture (\cite{br}).
%(which will call also FJC).
It states that the assembly map
$$H^G_n(E_{\V\C}G; {\bK}_{\A}) \to H^G_n(E_{\A ll}G; {\bK}_{\A}) = H^G_n(pt; {\bK}_{\A})$$
is an isomorphism. If a group satisfies the Conjecture, we say that the group satisfies the
FJC. We say that a group $G$ satisfies the FJCw if the wreath product $G{\wr}H$ satisfies the FJC for each
finite group $H$ and with coefficients.

We need the following basic facts:

\begin{rem}\label{rem-properties}
\begin{enumerate}
\item Word hyperbolic groups satisfy the $K$-FJCw (\cite{blr}). 
\item CAT(0)-groups satisfy the $K$-FJCw  (\cite{we}).
\item Strongly poly-free groups or, more generally, weak strongly poly-surface groups
satisfy the $K$-FJCw (\cite{rou}).
%\item Let $G \xrightarrow{p} K$ be an epimorphism. Assume that $K$ and $p^{-1}(V)$, for
%$V$ a virtually cyclic subgroup of $K$ satisfy the $K$-FJCw with coefficients. Then $G$ satisfies the $K$-FJC
%with coefficients.
\item Let $G$ satisfies the $K$-FJCw and 
$$1 \to G \to K \to H \to 1$$
be an exact sequence with $H$ finite. Then $K$ satisfies the $K$-FJCw (\cite{rou}).
\end{enumerate}
\end{rem}

In \cite{ba} it was shown that, for a ring $R$, the relative map 
$$H^G_n(E_{\F}G; {\bK}R^{-{\infty}}) \to H^G_n(E_{\V \C}G; {\bK}R^{-{\infty}})$$
is a split injection. Also, in \cite{dqr} it was shown that the natural map
$$H^G_n(E_{\F\B\C}G; {\bK}R^{-{\infty}}) \to H^G_n(E_{\V \C}G; {\bK}R^{-{\infty}})$$
is an isomorphism. Taking the corresponding cockerels, we have that
$$H^G_n(E_{\F}G \to E_{\V \C}G; {\bK}R^{-{\infty}}) \cong 
H^G_n(E_{\F}G \to E_{\F\B\C}G; {\bK}R^{-{\infty}})$$
Now, let the group $G$ satisfy the condition ${\M}_{{\F}\subset {\F\B\C}}$ of \cite{lw}, which states that
every infinite group in ${\F\B\C}$ is contained in a unique maximal group in ${\F\B\C}$. Then
$$\bigoplus_{V\in {\M}}H_n^{N_G(V)}(E_{\F}N_G(V) \to E_1W_G(V); {\bK}R^{-{\infty}}) \xrightarrow{\cong}  H^G_n(E_{\F}G \to E_{\F\B\C}G; {\bK}R^{-{\infty}}).$$
Here $\M$ is a set of representatives of the conjugacy classes of the maximal infinite groups in 
$\F\B\C$ and $W_G(V)=N_G(V)/V$ is the Weyl group of $V$ (\cite[Corollary 6.1]{lw})
Remark 6.2 in \cite{lw} implies that there is a spectral sequence
\begin{equation}\label{eq-spectral1}
\begin{array}{ll}
E^2_{p,q} & =  H_p^{W_G(V)}(E_1W_G(V); H^V_q(E_{\F}N_G(V) \to \{pt\}); {\bK}R^{-{\infty}})
\Rightarrow \\
&H_{p+q}^{N_G(V)}(E_{\F}N_G(V) \to E_1W_G(V); {\bK}R^{-{\infty}})
\end{array}
\end{equation}
That is obtained by choosing $X = E_1W_G(V)$ and noticing that $E_{\F}N_G(V) {\times}E_1W_G(V)$ is a space of type $E_{\F}N_G(V)$ with the diagonal action. Also,
Example 6.3 in \cite{lw} implies that, if $V = F{\rtimes}\bbZ$ then $H^V_q(E_{\F}N_G(V) \to \{pt\})$
is the non-connective version of Farrell's  twisted Nil-term. Thus the spectral sequence becomes:
\begin{equation}\label{eq-spectral}
E^2_{p,q} =  H_p^{W_G(V)}(E_1W_G(V); \mathbf{Nil}_R)\Rightarrow H_{p+q}^{N_G(V)}(E_{\F}N_G(V) \to E_1W_G(V); {\bK}R^{-{\infty}})
\end{equation}

\begin{rem}\label{rem-hol}
In \cite{mepr}, it was shown that:
\begin{enumerate}
\item The finite subgroups of $\text{Aut}(F_2)$ and ${\hol}(F_2)$ are isomorphic to ${\bbZ}/2{\bbZ}$,
${\bbZ}/3{\bbZ}$, ${\bbZ}/4{\bbZ}$, $D_2$ and $D_4$. The maximal finite subgroups are 
${\bbZ}/3{\bbZ}$ and $D_4$. From the construction, the same is true for ${\h}_{(n)}$ for
all $n \ge 1$.
\item Although this is not explicitly shown in \cite{mepr}, up to isomorphism, there are various infinite ${\F}{\B}{\C}$ subgroups of ${\hol}(F_2)$ which
are isomorphic to ${\bbZ}/2{\bbZ}{\times}{\bbZ}$. In fact, the construction of ${\h}_{(n)}$  shows that 
the infinite ${\F}{\B}{\C}$ subgroups are $F{\times}{\bbZ}$ where $F < {\h}_{(n)}$ is finite. The
subgroup $\bbZ$  is a subgroup in the factors that are complementary to ${\hol}(F_2)$. 
That is because each element that belongs to 
$F_i < \h_{(n)}$,
$3 \le i \le n$, commutes with the subgroups of $\hol(F_2)$.

Thus the maximal infinite ${\F}{\B}{\C}$ subgroups are of the following types: ${\bbZ}/2{\bbZ}{\times}{\bbZ}$,
${\bbZ}/3{\bbZ}{\times}{\bbZ}$ and $D_4{\times}{\bbZ}$.
%\item Every such subgroup is contained into a maximal cyclic by finite subgroup. This is an 
%immediate consequence
%of the fact that $\aut(F_2)$ decomposes as an amalgamated free product with maximal elements %of finite order.
\end{enumerate}
\end{rem}

We will show that certain mapping tori of the free groups that are contained in ${\h}_{(n)}$ are CAT(0)-groups. Notice that,
by the work of Brady \cite{brady}, all mapping tori $F_2\rtimes\bbZ$ contained in $F_2\rtimes \aut(F_2)$ are CAT(0).
We show that this is true for all mapping tori $F_{n+1}\rtimes\bbZ$ contained in $F_{n+1}\rtimes E(\h_{(n-1)})$.

%\begin{lem}\label{lem-npc}
%The group $F_2{\rtimes}\bbZ$ is a CAT(0)-group. More precisely, there is a$F_2{\rtimes}\bbZ$-free CAT(0)-space $X$ and a geodesic $\gamma$ so that the
%class of $\gamma$ in $X/F_2{\rtimes}\bbZ$ represents the generator of $\bbZ$.
%\end{lem}
%
%\begin{proof}
%Let $\bbZ = {\langle}t{\rangle}$. Let ${\phi}: F_2 \to F_2$ be the automorphism that determines the $t$-action on $F_2$. Let $S$ be the genus $1$ surface with one boundary component.
%As in \cite{mepr}, there is a diffeomorphism $f: S \to S$ so that the induced map $f_*$ on the fundamental groups induces the homomorphism 
%${\phi}$ in $Out(F_2)$. We can assume that $f$ fixes a point $x$ near the boundary. Then the mapping torus $T(f)$ admits a metric of non-positive curvature which
%extends the standard flat metric near the boundary. The universal cover $\widetilde{T(f)} = D(\tilde{f})$, where $\tilde{f}: \tilde{S} \to \tilde{S}$ is the
%lift of $f$ to the universal cover of $S$, and $D(\tilde{f})$ is the infinite mapping telescope of $\tilde{f}$. Then, by construction, $\{x\}{\times}\bbR \to D(\tilde{f})$
%is a geodesic. Also the image of this geodesic in $T(f)$ represents (up to base points) the element $t$ in ${\pi}(T(f))$. So we can choose $X = D(\tilde{f})$ to
%complete the proof.
%\end{proof}

\begin{prop}\label{prop-npc}
Let $G_n=F_{n+1}{\rtimes}\bbZ < F_{n+1}\rtimes E({\h}_{(n-1)})$. Then $G_n$ is CAT(0). 
%More precisely, there is a$G_n$-free CAT(0)-space $X$ and a geodesic $\gamma$ so that the class of $\gamma$ in $X/G_n$ represents the generator of $\bbZ$.
\end{prop}

\begin{proof}
Set $G_n = F_{n+1}{\rtimes}\bbZ < {\h}_{(n)}$. By definition, the $\bbZ$-action on $F_{n+1}$ is from an element of ${\h}_{(n-1)}$. That means that in the first two
generators it is an automorphism of the free group they generate and on the other generators is conjugation by words on the previous generators. We will use
induction. For $n = 1$, $G_1 = F_2{\rtimes}{\bbZ}$ which is CAT(0) (\cite{brady}).

For general $n$, let $F_{n+1} = {\langle}x_1, x_2, \dots, x_{n+1}{\rangle}$.
Notice that ${\h}_{(n-1)} = F_n{\rtimes}{\h}_{(n-2)}$. Then every $g \in {\h}_{(n-1)}$ can be written as $g = g_1g_2$, with $g_1\in F_n$ and
$g_2 \in {\h}_{(n-2)}$. Then the embedding ${\h}_{(n-1)}$ in $\aut(F_{n+1})$ sends $g$ to $\tilde{g}$ with:
$$\tilde{g}(x_i) = g_2(x_i), \; i = 1, 2, \ldots, n, \;\; \tilde{g}(x_{n+1}) = g_1x_{n+1}g_1^{-1}.$$
Then 
$$G_n = F_{n+1}{\rtimes}_{\tilde{g}}\bbZ ={\langle}t, x_1, x_2, \dots, x_{n+1}: tx_it^{-1} = g_2(x_i), \; i = 1, 2 \ldots, n, \; tx_{n+1}t^{-1} = g_1x_{n+1}g_1^{-1}{\rangle}$$
with $g_1$ a word in $x_i$, $i = 1, 2, \ldots, n$. Setting  ${\alpha} = g_1^{-1}t$ and solving for $t$, we get
$$G_n = {\langle}{\alpha}, x_1, x_2, \dots x_{n+1}: {\alpha}x_i{\alpha}^{-1} = g_1^{-1}g_2(x_i)g_1, \; i = 1, 2, \dots, n{\rangle}*_{\bbZ}{\langle}{\alpha}, x_{n+1}:
[{\alpha}, x_{n+1}] = 1{\rangle},$$
where $\bbZ = {\langle}{\alpha}{\rangle}$. Then $G_n = H*_{\bbZ}{\bbZ}^2$. To characterize the group $H$, we set $\beta = g_1{\alpha}$, 
$$H = {\langle}{\beta}, x_1, x_2, \dots, x_{n}: {\beta}x_i{\beta}^{-1} = g_2(x_i), \; i = 1, 2, \dots, n{\rangle}.$$
If $n = 2$, then $H = F_2{\rtimes}{\bbZ} = G_1$. If $n > 2$, 
then $g_2 \in {\h}_{(n-2)}$. Thus $H = G_{n-1}$. So $G_n = G_{n-1}*_{\bbZ}{\bbZ}^2$ which is CAT(0) by induction and \cite{brha}, Part II, Proposition 11.19.
\end{proof}

The following is in \cite{we}.

\begin{cor}\label{cor-npc}
The groups $G_n$ in Proposition \ref{prop-npc} satisfy the $K$-FJCw.
\end{cor}

In \cite{mepr}, it was shown that the only infinite, virtually cyclic subgroup of Type (I) in $\hol(F_2)$ is $\bbZ{\times}{\bbZ}/
2{\bbZ}$. 

\begin{cor}\label{cor-ic}
The group $F_n{\rtimes}(\bbZ{\times}{\bbZ}/2{\bbZ}) < {\h}_{(n-1)}$ satisfies the $K$-FICw.
\end{cor}

\begin{proof}
In \cite{we1}, it was shown that CAT(0)-groups satisfy the $K$-FJCw. That means that finite
extensions of CAT(0)-groups satisfy the $K$-FJCw. The result follows.
\end{proof}

\begin{rem}\label{rem-npc}
In the Appendix, we will show that certain groups that appear in Corollary \ref{cor-ic} are CAT(0).
\end{rem}

We are also able to prove the following.

\begin{prop}\label{general}
Let $D$ be a finite subgroup of $\aut(F_2)$ and $E(D)$ its image in $\h_{(n)}$. Then there are infinite cyclic-by-finite subgroups of the form $\bbZ\times E(D)$ in $\hol(F_n)\setminus \aut(F_n)$,
$n\ge 3$. Moreover, every  $G = F_m\rtimes (\bbZ\times E(D))$ in $\h_{(m-2)}$ is CAT(0) for all $m > n$.
\end{prop}

\begin{proof}
Assume that $F_n=\< x_1,\ldots, x_n\>$. By definition of $E$, for all $\phi\in E(D)$ we have that $\phi(x_i)=x_i$ for
every $i>2$. Hence, in $\hol(F_n)$, $\< x_i\>$ with $i>2$ commutes with $\phi$ for all $\phi\in E(D)$ and so
$\< x_i, E(D)\>$ is isomorphic to $\bbZ\times E(D)$ for every $i>2$.

Let us now see $\bbZ\times E(D)$ as a subgroup of $\hol(F_n)$, $\bbZ=\< x_n\>$, $n\ge 3$ and embed it in $\aut(F_{n+1})$. Then the action of $\bbZ=\<\xi_{x_n}\>$ on $x_{n+1}$ is conjugation by $x_n$ and is trivial on 
every other generator of $F_{n+1}$.  Then $G$ has a presentation of the following form
$$G=\< x_1,\ldots, x_{n+1}, \xi_{x_n}, E(D) \mid [\xi_{x_n},E(D)]=1, [\xi_{x_n}, x_j]=1 \ \ \ \forall j< n+1,$$  $$\xi_{x_n}  x_{n+1}\xi_{x_n}^{-1}=x_n x_{n+1}x_n^{-1}, [x_i, E(D)]=1 \ \ \forall i>2\>.$$

Set $z=x_n^{-1}\xi_{x_n}$, and get rid of $x_n$ to get the presentation
$$G=\< x_1,\ldots, x_{n-1}, z, x_{n+1}, E(D) \mid [\xi_{x_n},E(D)]=1, [z,x_j]=1 \ \ \ \forall j< n+1, j\neq n$$  
$$[\xi_{x_n}, z]=1, [x_i, E(D)]=1 \ \ \forall i>2, i\neq n, [z,E(D)]=1\>.$$ 

Now decompose $G$ as an amalgamated free product $G=G_1 *_{\bbZ \times E(D)} G_2$ with 
$$G_1=\< x_1, x_2, \xi_{x_n}, E(D)\>\cong F_2\rtimes (\bbZ\times E(D)),$$
$$G_2=\< x_3, \ldots, x_{n-1}, x_{n+1}, z, \xi_{x_n}, E(D)\>\cong \< x_3,\ldots, x_{n_1},x_{n+1}, z, \xi_{x_n}\>\times E(D).$$
Now notice that the subgroup generated by $\< x_3,\ldots, x_{n_1},x_{n+1}, z, \xi_{x_n}\>$ has a presentation
$$\< x_3,\ldots, x_{n_1},x_{n+1}, z, \xi_{x_n}\mid [\xi_{x_n},z]=1, [\xi_{x_n}, x_j]=1 \ \ \forall j<n, [z,x_{n+1}]=1 \>$$
which makes it a right-angled Artin group and so is CAT(0) by \cite{dj}.
Thus $G$ is CAT(0) by \cite{brha}, Part II, Proposition 11.19.
\end{proof}

\begin{rem}
In the last proposition, we showed that the group $G$ is CAT(0) and thus it satisfies the $K$-FJCw. But we can use Proposition \ref{prop-npc} to show directly that $G$ satisfies the $K$-FJCw.
That is done as in Corollary \ref{cor-ic}.
\end{rem}

\begin{rem}\label{rem-max}
Note the following two properties of the groups described in  Lemma \ref{general}.
\begin{enumerate}
\item Every such subgroup is contained into a maximal cyclic by finite subgroup. This is an immediate consequence
of the fact that $\aut(F_2)$ decomposes as an amalgamated free product with maximal elements of finite order.
\item The normalizer of every maximal such subgroup coincides with the normalizer of its finite subgroup in $\aut(F_2)$. 
\end{enumerate}
\end{rem}

Now let us introduce some notation from \cite{mepr}. The group $\aut(F_2)$ admits a presentation of
the form
$$\< p, x,y,\tau_a,\tau_b\mid x^4=p^2=(px)^2=1, (py)^2=\tau_b, x^2=y^3\tau_b^{-1}\tau_a,$$
$$p^{-1}\tau_ap=x^{-1}\tau_ax=y^{-1}\tau_ay=\tau_b, p^{-1}\tau_bp=\tau_a, x^{-1}\tau_bx=\tau_a^{-1}, y^{-1}\tau_by=\tau_a^{-1}\tau_b\>$$
where $\tau_a,\tau_b$ are the inner automorphism of $F_2$ corresponding to $a,b$
respectively.
Moreover, any element of $\aut(F_2)$ can be written uniquely
in the form $p^{r}u(x,y)x^{2s}w(\tau_a,\tau_b)$ where $r,s\in \{0,1\}$, $w(\tau_a,\tau_b)$ is a reduced
word in $\inn(F_2)$ and $u(x,y)$ is a reduced word where $x,y,y^{-1}$ are the only
powers of $x,y$ appearing (see \cite{meskin,mks}). 

Moreover, a presentation for $\gl$ is
given by
$$\gl=\< P,X,Y\mid X^4=P^2=(PX)^2=(PY)^2=1, X^2=Y^3\>$$ and $\aut(F_2)$ maps homomorphically
onto $\gl$ by $p\mapsto P$, $x\mapsto X$, $y\mapsto Y$, $\tau_a,\tau_b\mapsto 1$.

\begin{lem}\label{lem-D4}
Let $D_4$ be the subgroup of $\aut(F_2)$ generated by $\<p,x\>$. Then the normalizer $N_{\aut(F_2)}(D_4)$ of $D_4$ 
in $\aut(F_2)$ is $D_4$ itself.
\end{lem}

\begin{proof}
Let  $p^ru(x,y)x^{2s}w(\tau_a,\tau_b)$ be an element of $\aut(F_2)$ that belongs to the normalizer of $D_4$. Then it necessarily conjugates
elements of order 4 to elements of order 4. But the only elements of order 4 in $\aut(F_2)$ are conjugates of $x^{\pm 1}$ (see \cite{meskin}). Hence we have the following relation
\begin{equation}\label{rel1}
p^rux^{2s}w\cdot x\cdot w^{-1}x^{-2s}u^{-1}p^{-r}=x^{\pm 1}
\end{equation} or equivalently
$$ux^{2s}w\cdot x\cdot w^{-1}x^{-2s}u^{-1}=p^rx^{\pm 1}x^{-r}$$ i.e.
$$ux^{2s}w\cdot x\cdot w^{-1}x^{-2s}u^{-1}=x^{\pm 1}.$$ Now project this relation to GL$_2(\bbZ)$. It reduces to
$U(X,Y)X^{2s}XX^{-2s}U^{-1}=X^{\pm1}$ or $UXU^{-1}=X^{\pm 1}$. This last relation implies that
$U=X$ and therefore $u=x$ since the projection maps $y$ to $D_6\setminus D_2$ and $x$ to $D_4\setminus D_2$, and 
therefore $U$ and $X$ freely generate a free group. 

 Thus  (\ref{rel1}) reduces to $w(\tau_a,\tau_b)xw^{-1}(\tau_a,\tau_b)=x^{\pm1}$ which implies that $w=1$. Hence
 the only words of $\aut(F_2)$ that normalize $x$ are of the form $p^rx^s$, hence $N_{\aut(F_2)}(D_4)=D_4$.
\end{proof}

\begin{cor}\label{cor-D4}
The normalizer $N_{{\h}_{(n)}}(D_4{\times}{\bbZ}) = D_4{\times}{\bbZ}$.
\end{cor}

\begin{proof}
It follows from \ref{lem-D4} and Remark \ref{rem-max}.
\end{proof}

%%%%%%%%%%%%%%%%%%%%%%%%%%%%%%%%%%%%%%%%%%%%%%%%%%
\section{The Lower $K$-theory for $\h_{(n)}$}
%%%%%%%%%%%%%%%%%%%%%%%%%%%%%%%%%%%%%%%%%%%%%%%%%%%

The algebraic calculations of the previous section allows us to prove the $K$-theoretic
Isomorphism Conjecture for the groups ${\h}_{(n)}$, using induction.

We start with the first case.

\begin{prop}\label{prop-F2}
The groups $\aut(F_2)$ and $\hol(F_2)$ satisfy the $K$-FJCw.
\end{prop}

\begin{proof}
In \cite{mepr}, it was shown that the groups $\aut(F_2)$ and $\hol(F_2)$ are strongly poly-free groups.
The result follows from \cite{rou}.
\end{proof}

%\begin{lem}\label{lem-finite}
%\begin{enumerate}
%\item The finite subgroups of $\h_{(n)}$ are the same as the finite subgroups of $\hol(F_2)$.
%\item The infinite cyclic-by-finite subgroups of $\h_{(n)}$ are of the form ${\bbZ}{\times}F$,
%where $F$ is a finite subgroup of $\hol(F_2)$.
%\end{enumerate}
%\end{lem}
%
%\begin{proof}
%The first part is obvious. For the second part, notice that each element that belongs to 
%$F_i < \h_{(n)}$,
%$3 \le i \le n$, commutes with the subgroups of $\hol(F_2)$.
%\end{proof}

\begin{thm}\label{thm-iso}
The groups ${\h}_{(n)}$ satisfy the $K$-FJCw.
\end{thm}

\begin{proof}
We will use induction on $n$. For $n = 0$, ${\h}_{(0)} = F_2$, for which the fibered isomorphism conjecture holds (\cite{blr}). For $n = 1$, the result in Proposition \ref{prop-F2}. So we assume that
$n > 1$. Then we have an exact sequence
$$1 \to F_{n+1} \to {\h}_{(n)} \xrightarrow{p} {\h}_{(n-1)} \to 1.$$
We assume that the $K$-FJCw holds for ${\h}_{(n-1)}$. So, it suffices to show
that the $K$-FJCw  holds for $p^{-1}(V)$ where $V$ is in $\F\B\C$ of
${\h}_{(n-1)}$. There are two cases:
\begin{enumerate}
\item Let $V < {\h}_{(n-1)}$ be finite. Then $p^{-1}(V) \cong F_{n+1}{\rtimes}V$ is a hyperbolic
group. Thus it satisfies the $K$-FJCw (\cite{blr}, \cite{we1}).
\item Let $V$ is an infinite ${\F}{\B}{\C}$ subgroup of ${\h}_{(n-1)}$. There are three cases:
\begin{enumerate}
\item Let $V \cong \bbZ$. Proposition \ref{prop-npc} implies that $p^{-1}(V)$ is CAT(0). 
Thus it satisfies the $K$-FJCw (Corollary \ref{cor-npc}).
\item Let $V \cong {\bbZ}{2\bbZ}{\times}{\bbZ} < {\hol}(F_2)$. 
The result follows from Corollary \ref{cor-ic}.
\item Let $V \cong D{\times}{\bbZ}$ where $D$ is finite and $D{\times}{\bbZ} \in {\h}_{(n-1)} 
\setminus {\hol}(F_2)$. The Proposition \ref{general} implies that $p^{-1}(V)$ 
is CAT(0). Thus it satisfies the $K$-FJCw.
\end{enumerate} 
\end{enumerate}
The result follows from Remark \ref{rem-properties}.
\end{proof}

%Let $\underline{E}{\h}_{(n)}$ be the classifying space for the class $\F$ of finite subgroups of 
%${\h}_{(n)}$
%and $\underline{\underline{E}}{\h}_{(n)}$ be the classifying space for the class $\F\B\C$ of infinite cyclic-by-finite
%subgroups. Let $\mathbf{K}_{\bbZ}$ be the $K$-theory spectrum with coefFJCients in $\bbZ$. 

Using Theorem \ref{thm-iso}, we calculate the lower $K$-theory of ${\h}_{(n)}$.

\begin{thm}
The groups $K_q(\bbZ\h_{(n)}) = 0$, for $q \le -1$. For the reduced  $K$-groups in the other
dimensions, we have, for  $q = 0, 1$,
$$\tilde{K}_q(\bbZ\h_{(n)}) = \left\{
\begin{array}{cl}
0, & n = 0, 1\\
NK_q(\bbZ D_4){\oplus}NK_q(\bbZ D_4), & n \ge 2
\end{array}
\right.$$
\end{thm}

\begin{proof}
We assume $n \ge 1$, because $\h_{(0)} = F_2$, which is a hyperbolic group,  and 
$\h_{(0)} = \text{Hol}(F_2)$ and the result was proved in \cite{mepr}.
Since the groups ${\h}_{(n)}$ satisfy the $K$-FJCw we have that
$$\begin{array}{ll}
&K_q({\bbZ}\h_{(n)}) \cong H^{\h_{(n)}}_q(E_{\F\B\C}\h_{(n)}; {\bK}{\bbZ}^{-{\infty}})  \cong \\[2ex]
&\displaystyle{
H^{\h_{(n)}}_q( E_{\F}\h_{(n)}; {\bK}{\bbZ}^{-{\infty}}) \oplus
\bigoplus_{V \in {\M}} H_q^{N_{\h_{(n)}}V}(E_{\F}N_{\h_{(n)}}V \to E_1W_{\h_{(n)}}V; 
{\bK}{\bbZ}^{-{\infty}})},
\end{array}$$
where $\M$ is a set of representatives of the conjugacy classes of maximal infinite groups in 
$\F\B\C$. 

The calculations of \cite{mepr} show that $H^{\h_{(n)}}_n( E_{\F}\h_{(n)}; {\bK}{\bbZ}^{-\infty}) = 0$, for
$n \le 1$. For each of the summands, there is a spectral sequence
$$E^2_{p,q} = H_p^{W_{\h_{(n)}}V}(E_1W_{\h_{(n)}}V; H_q^V(E_{\F}V \to pt; \bK{\bbZ}^{-\infty})) \Rightarrow
H_{p+q}^{N_{\h_{(n)}}V}(E_{\F}N_{\h_{(n)}}V \to E_1W_{\h_{(n)}}V; {\bK}{\bbZ}^{-\infty})
$$
In Remark \ref{rem-hol} it was shown that the maximal infinite groups 
in $\h_{(n)}$ are of the following types 
$({\bbZ}/2\bbZ){\times}{\bbZ}$, $({\bbZ}/3\bbZ){\times}{\bbZ}$, $D_4{\times}{\bbZ}$ and there is only one conjugacy class for
each of the groups ${\bbZ}_3{\times}{\bbZ}$ and $D_4{\times}{\bbZ}$ . 
\begin{enumerate}
\item When $V$ is ${\bbZ}_2{\times}{\bbZ}$ or ${\bbZ}_3{\times}{\bbZ}$, then 
$H_q^V(E_{\F}V \to pt; \bK_{\bbZ}) = 0$, $q \le 1$ because the Nil-groups of the two cyclic
groups vanish. Thus, for these groups, (Spectral sequence (\ref{eq-spectral}))
$$H_n^{N_{\h_{(n)}}V}(E_{\F}N_{\h_{(n)}}V \to E_1W_{\h_{(n)}}V; {\bK}{\bbZ}^{-\infty}) = 0, 
\;\; i \le 1$$
\item For $V = D_4{\times}{\bbZ}$, we have that $N_{\h_{(n)}}V = V$ (Corollary \ref{cor-D4})
and the spectral sequence (Spectral sequence (\ref{eq-spectral1}))
reduces to the isomorphism for $q \le 1$
$$H_q^{N_{\h_{(n)}}V}(E_{\F}N_{\h_{(n)}}V \to E_1W_{\h_{(n)}}V; {\bK}{\bbZ}^{-\infty}) \cong H_q^V(E_{\F}V \to pt; \bK{\bbZ}^{-\infty}) \cong NK_q({\bbZ}D_4) {\oplus} NK_q({\bbZ}D_4).
$$
It is known that $NK_q(\bbZ D_4) = 0$ for $q \le -1$ and it is infinitely generated for $q = 0, 1$ 
(\cite{wei}).
\end{enumerate}
Combining the above information, we have that ($n \ge 2$):
\begin{enumerate}
\item $K_i(\bbZ\h_{(n)}) = 0$, $i \le -1$.
\item ${\tilde K}_0(\bbZ\h_{(n)}) \cong NK_0(\bbZ D_4) {\oplus} NK_0(\bbZ D_4)$.
\item $Wh(\h_{(n)}) \cong NK_1(\bbZ D_4) {\oplus} NK_1(\bbZ D_4)$.
\end{enumerate}
\end{proof}

\begin{rem}
In \cite{wei} it was shown that $NK_0(\bbZ D_4)$ is isomorphic to the direct sum of infinite free
$\bbZ_2$-module with a countably infinite free $\bbZ_4$-module. Also, $NK_1(\bbZ D_4)$ is a countably infinite torsion group of exponent $2$ or $4$.
\end{rem}

%%%%%%%%%%%%%%%%%%%%%%%%%%%%%%%%%%%%%%
\section{Concluding Remarks}
%%%%%%%%%%%%%%%%%%%%%%%%%%%%%%%%%%%%%%

In general, if the group $G$ is linear (i.e. it admits a faithful finite dimensional real or complex representation) then 
the $K$-FJCw can be proved for $G$. The problem is that $\aut(F_n)$ is not linear for $n \ge 3$.
The group that was used to show that $\aut(F_n)$ is not linear is the Formanek-Procesi group, $FP$ (\cite{fp}).
That group given by a split extension:
$$1 \to F_3 \xrightarrow{f} FP \xrightarrow{p} F_2 \to 1$$
and has presentation:
$$FP = {\langle} {\alpha}_1, {\alpha}_2, {\alpha}_3, {\phi}_1, {\phi}_2\; : \: {\phi}_i{\alpha}_j{\phi}_i^{-1} = {\alpha}_j, \;
{\phi}_i{\alpha}_3{\phi}_i^{-1} = {\alpha}_3{\alpha}_i, \; i, j = 1, 2{\rangle}.$$
In \cite{fp}, it was shown that $FP$ is not linear and $FP < \aut(F_3)$.
On the other hand, it is obvious that the group $G$ is not word hyperbolic (since it contains a subgroup isomorphic to $\bbZ\times\bbZ$) and not known if it is CAT(0).  
The point here is that $FP$ has ``enough'' CAT(0)-subgroups so that it satisfies the $K$-FJCw.

\begin{prop}\label{fp}
The group $FP$ satisfies the $K$-FJCw.
\end{prop}
 
\begin{proof}
All the virtually cyclic subgroups of $F_2$ are infinite cyclic. Let $V$ be such a group that is generated by a word on
${\phi}_1$ and ${\phi}_2$ and their inverses. Then the action of the generator on $F_3$ fixes the first two generators
and sends ${\alpha}_3$ to the element ${\alpha}_3c$ where $c$ is a word in ${\alpha}_1$, ${\alpha}_2$ and their
inverses. Then $F_3{\rtimes}V$ is a CAT(0)-group from Theorem 4.4 in \cite{sa}. Thus, it satisfies the $K$-FJCw.
Therefore $FP$ satisfies the $K$-FJCw.
\end{proof}

%%%%%%%%%%%%%%%%%%%%%%%%%%%%%%%%%%%%%%%%%%%%%%%%%%%%%%
\section{Appendix}
%%%%%%%%%%%%%%%%%%%%%%%%%%%%%%%%%%%%%%%%%%%%%%%%%%%%%%
We will show the result that was stated in Remark \ref{rem-npc}, that certain groups of type
$F_n{\rtimes}(\bbZ\times {\bbZ}/2{\bbZ}) < {\h}_{(n-1)}$ are CAT(0).

Our investigation, similar to the one in \cite[Proposition 3.2]{mepr} shows that subgroups of $\aut(F_2)$ isomorphic
to $F_2\rtimes (\bbZ\times\bbZ/2\bbZ)$  occur when the action of $\bbZ/2\bbZ$ on $F_2$ is of the following form
$$t(x_1) = x_1^{-1}, \;\; t(x_2) = x_2\ \  \mbox{or} \ \ t(x_1)=x_1,\;\; t(x_2)=x_2^{-1}.$$
We show that in all the above cases these subgroups are CAT($0$).

\begin{lem}\label{lem-aut}
Let ${\bbZ}/2{\bbZ} \cong {\langle}t_1{\rangle} < \aut(F_2)$ be such that $t_1(x_1) = x_1^{-1}$  and $t_1(x_2) = x_2$, as above. Let $\bbZ \cong {\langle}t_2{\rangle} <
\hol(F_2)$ so that $\bbZ{\times}{\bbZ}/2{\bbZ} < \hol(F_2)$. Then we have the following cases:
\begin{enumerate}
\item If $t_2 \notin \aut(F_2)$, then $t_2 = x_2^k \in F_2$, $k \in \bbZ$, $k\neq 0$.
\item If $t_2 \in \aut(F_2)$, then we have $t_2(x_1) = x_2^kx_1^{{\pm}1}x_2^{-k}$, $t_2(x_2) = x_2^{{\pm}1}$, $k\neq0$ (four cases).
%two cases:
%\begin{enumerate}
%\item $t_2(x_1) = x_2^kx_1x_2^{-k}$, $t_2(x_2) = x_2$.
%\item $t_2(x_1) = x_2^kx_1^{-1}x_2^{-k}$, $t_2(x_2) = x_2$.
%\end{enumerate}
\end{enumerate}
\end{lem}

\begin{proof}
First, we assume $t_2 \notin \aut(F_2)$. Because $t_1$ and $t_2$ commute, $t_1$ acts trivially on $t_2$. Let $w(x_1, x_2)$ be the
word in $F_2$ representing $t_2$. Then, $w(x_1^{-1}, x_2) = w(x_1, x_2)$. That means that $x_1$ does not appear in $w(x_1, x_2)$. Thus $w(x_1, x_2) = x_2^k$,

Now let $t_2 \in \aut(F_2)$. Then $t_2(x_1) = w_1(x_1, x_2)$ and $t_2(x_2) = w_2(x_1, x_2)$. Since $t_1t_2 = t_2t_1$, we have that
$$w_1(x_1^{-1}, x_2) = w_1(x_1, x_2)^{-1}, \;\; w_2(x_1^{-1}, x_2) = w_2(x_1, x_2).$$
As before, the second relation implies that the word $w_2 = x_2^k$, $k \in \bbZ$.
Looking at the first relation, we get that $w_1 = cx_1^{\ell}c^{-1}$, ${\ell}\in \bbZ$. 
But $w_1$ and $w_2$ must be a generating set for $F_2$. That means that $k, {\ell}\in \{{\pm}1\}$. Also,
$$\begin{array}{lll}
(t_1{\circ}t_2)(x_1) & = &t_1(c(x_1, x_2)x_1^{{\pm}1}c(x_1, x_2)^{-1}) = c(x_1^{-1}, x_2)x_1^{{\mp}1}c(x_1^{-1}, x_2)^{-1}\\
(t_2{\circ}t_1)(x_1) & = & t_2(x_1^{-1}) =  c(x_1, x_2)x_1^{{\mp}1}c(x_1, x_2)^{-1}
\end{array}$$
Since $t_1\circ t_2 = t_2\circ t_1$, $c(x_1, x_2) = c(x_1^{-1}, x_2)$ and thus $c = x_2^k$, $k \in \bbZ$.
\end{proof}

\begin{lem}\label{lem-aut-F2}
Let $G = F_2{\rtimes}({\bbZ}{\times}{\bbZ}/2{\bbZ})$ where the generator $t_1$ of $\bbZ/2\bbZ$ acts as
$$t_1(x_1) = x_1^{-1}, \;\; t_1(x_2) = x_2$$
and the generator $t_2$ of $\bbZ$ acts as
$$t_2(x_1) = x_2^kx_1^{{\pm}1}x_2^{-k}, \;\; t_2(x_2) = x_2^{{\pm}1}.$$
Then $G$ is a CAT(0)-group.
\end{lem}

\begin{proof}
The group $G$ has the following presentation:
$${\langle}t_1, t_2, x_1, x_2:\; t_1x_1t_1^{-1} = x_1^{-1}, t_1x_2t_1^{-1} = x_2, t_1^2 = [t_1, t_2] = 1, 
 t_2x_1t_2^{-1} =  x_2^kx_1^{{\pm}1}x_2^{-k}, \; t_2x_2t_2^{-1} = x_2^{{\pm}1}{\rangle}.$$
We set ${\xi} = x_2^{-k}t_2$. Then the presentation becomes:
$${\langle}t_1, x_1, x_2, {\xi}:\;  t_1x_1t_1^{-1} = x_1^{-1}, t_1x_2t_1^{-1} = x_2, t_1^2 =  [t_1, {\xi}] = 1, 
{\xi}x_1{\xi}^{-1} =  x_1^{{\pm}1}, \; {\xi}x_2{\xi}^{-1} = x_2^{{\pm}1}{\rangle}.$$
Now we consider four cases:

\vspace{12pt}
\noindent
\underline{Case 1}. In this case,
$$G = {\langle}t_1, x_1, x_2, {\xi}:\; t_1x_1t_1^{-1} = x_1^{-1}, t_1x_2t_1^{-1} = x_2, t_1^2 = [t_1, {\xi}] = 1,
{\xi}x_1{\xi}^{-1} =  x_1, \; {\xi}x_2{\xi}^{-1} = x_2{\rangle}.$$ 
Now, let $L_1 = {\langle} x_2, t_1, \xi :   t_1^2 = [t_1, x_2] = [t_1, {\xi}] = [x_2, \xi] = 1{\rangle} < G$ which is isomorphic to ${\bbZ}^2{\times}{\bbZ}/2{\bbZ}$
and thus it is CAT(0). Also, let $L_2 = {\langle} x_1, t_1, \xi :   t_1x_1t_1^{-1} = x_1^{-1}, 
{\xi}x_1{\xi}^{-1} = x_1, t_1^2 = [t_1, \xi ] = 1{\rangle}$. Then 
$$L_2 = {\langle}{\xi}{\rangle}{\times}{\langle} x_1, t_1: t_1x_1t_1^{-1} = x_1^{-1}{\rangle} \cong {\bbZ}{\times}D_{\infty}.$$
The infinite dihedral group is CAT(0)-group because it is a Coxeter group (\cite{mou}). Thus $L_2$ is a CAT(0)-group, as a direct product of CAT(0)-groups. Also,
$L = {\langle}t_1, {\xi} :\; t_1 ^2 = [t_1, {\xi}] = 1{\rangle} \cong {\bbZ}{\times}{\bbZ}/2{\bbZ}$. 
But then $G = L_1*_LL_2$ is CAT(0) (\cite{brha}, Part II, Corollary 11.19).

\vspace{12pt}
\noindent
\underline{Case 2}. We assume,
$$G = {\langle}t_1, x_1, x_2, {\xi}:\; t_1x_1t_1^{-1} = x_1^{-1}, t_1x_2t_1^{-1} = x_2, t_1^2 = [t_1, {\xi}] = 1,
{\xi}x_1{\xi}^{-1} =  x_1^{-1}, \; {\xi}x_2{\xi}^{-1} = x_2{\rangle}.$$ 
By setting $\xi_2=t_1\xi$ and rewritting the presentation we are back in Case I.

%Then, $L_1 \cong {\bbZ}^2{\times}{\bbZ}/2{\bbZ}$ and $L$ as before. Let 
%$$L_2 = {\langle} x_1, t_1, \xi :   t_1x_1t_1^{-1} = 
%{\xi}x_1{\xi}^{-1} = x_1^{-1}, t_1^2 = [t_1, \xi ] = 1{\rangle}$$
%We set ${\zeta} = {\xi}^{-1}t_1$ and the presentation becomes:
%$$L_2 = {\langle} x_1, t_1, \zeta :   t_1x_1t_1^{-1} = x_1^{-1}, 
%{\zeta}x_1{\zeta}^{-1} = x_1, t_1^2 = [t_1, \zeta ] = 1{\rangle} \cong {\bbZ}
%{\times}D_{\infty}.$$
%As in Case 1,  $G = L_1*_LL_2$ which is CAT(0).

\vspace{12pt}
\noindent
\underline{Case 3}. We assume,
$$G = {\langle}t_1, x_1, x_2, {\xi}:\; t_1x_1t_1^{-1} = x_1^{-1}, t_1x_2t_1^{-1} = x_2, t_1^2 = [t_1, {\xi}] = 1,
{\xi}x_1{\xi}^{-1} =  x_1, \; {\xi}x_2{\xi}^{-1} = x_2^{-1}{\rangle}.$$ 
We repeat the same method as before. Let 
$$L_1 = {\langle} x_2, t_1, \xi :   t_1^2 = [t_1, x_2] = [t_1, {\xi}] = 1,  {\xi}x_2 {\xi}^{-1} = x_2^{-1}{\rangle} \cong {\langle}t_1\mid t_1^2=1{\rangle}{\times}
{\langle}x_2, \xi :   {\xi}x_2 {\xi}^{-1} = x_2^{-1}{\rangle}$$
Therefore $L_1 \cong {\bbZ}/2{\bbZ}{\times}({\bbZ}{\rtimes}{\bbZ})$. Now the second group can be written as an HNN-extension ${\bbZ}*_{r}$ where $r$ is the
non-trivial automorphism of $\bbZ$. Then ${\bbZ}{\rtimes}{\bbZ}$ is CAT(0)-group (\cite{brha}, Part II, Corollary 11.22), and thus $L_1$ is CAT(0). Now $L$ and
$L_2$ are as in Case 1, and thus $G = L_1*_LL_2$ is CAT(0).

\vspace{12pt}
\noindent
\underline{Case 4}. We assume,
$$G = {\langle}t_1, x_1, x_2, {\xi}:\; t_1x_1t_1^{-1} = x_1^{-1}, t_1x_2t_1^{-1} = x_2, t_1^2 = [t_1, {\xi}] = 1,
{\xi}x_1{\xi}^{-1} =  x_1^{-1}, \; {\xi}x_2{\xi}^{-1} = x_2^{-1}{\rangle}.$$ 
Again set $\xi_2=t_1\xi$ and rewritte the presentation to arrive at Case 3.

%As expected, this case combines the calculations from the above cases. We define 
%$L_1$ as in Case 3 and $L_2$ as in Case 2. Again,
%$G = L_1*_LL_2$ which is CAT(0).
\end{proof}

\begin{prop}\label{prop-npc-Z2}
The group $F_n{\rtimes}(\bbZ{\times}{\bbZ}/2{\bbZ}) < {\h}_{(n-1)}$ is a CAT(0)-group.
\end{prop}

\begin{proof}
Let $t_1$ be the generator of ${\bbZ}/2{\bbZ}$ and $t_2$ the generator of ${\bbZ}$. We will consider two cases:

\vspace{12pt}
\noindent
\underline{Case 1}. Let $t_2 \in {\hol}(F_2)$ and $t_2 \notin \aut(F_2)$. From Lemma \ref{lem-aut}, Part (1), $t_2$ is an element $x_2^k$, $k \in \bbZ$, $k\neq0$. 
%Because $t_1$ and $t_2$ commute, $t_1$ acts trivially on $t_2$. Let $w(x_1, x_2)$ be the
%word in $F_2$ representing $t_2$. Then, $w(x_1^{-1}, x_2) = w(x_1, x_2)$. That means that $x_1$ does not appear in $w(x_1, x_2)$. Thus $w(x_1, x_2) = x_2^k$,
%for some $k \in \bbZ$.
%the word $t_2 \in F_2$ cannot contain the
%generator $x_1$. Thus $t_2$ will be  $x_2^{{\pm}1}$. We can assume that $t_2 = x_2$.
Then $G = F_n{\rtimes}(\bbZ{\times}{\bbZ}/2{\bbZ})$ has the following presentation:
$$\begin{array}{ll}
{\langle} x_1, x_2, \ldots, x_n, t_1, t_2: &  t_1^2 = 1, t_1x_1t_1^{-1} = x_1^{-1}, t_1x_it_1^{-1} = x_i, i=2,\ldots, n, \\
&t_2x_1t_2^{-1} = x_1, t_2x_2t_2^{-1} = x_2, t_2x_it_2^{-1} = x_2^kx_ix_2^{-k}, i=3,\ldots,n,  [t_1, t_2] = 1{\rangle}
\end{array}$$
We change the generators by setting ${\xi} = x_2^{-k}t_2$.  First notice that $[t_1, {\xi}] =1$ because $t_1$ commutes with $t_2$ and $x_2$.
Then the presentation becomes:
$$\begin{array}{ll}
{\langle} x_1, x_2, \ldots, x_n, t_1, \xi : &  t_1^2 = 1, t_1x_1t_1^{-1} = x_1^{-1}, t_1x_it_1^{-1} = x_i, i=2,\ldots, x_n, \\
&{\xi}x_1{\xi}^{-1} = x_2^{-k}x_1x_2^k, {\xi}x_i{\xi}^{-1} = x_i, i=2,\ldots,x_n, [t_1, \xi ] = 1{\rangle}
\end{array}$$
Set $K_1 = {\langle}t_1, \xi, x_3,\ldots, x_n : \; t_1^2 = [t_1, \xi] = [t_1, x_i] =[\xi,x_i] 1, i=3,\ldots,n {\rangle} < F_n{\rtimes}(\bbZ{\times}{\bbZ}/2{\bbZ})$ and it is isomorphic
to $\bbZ^{n-1}{\times}{\bbZ}/2{\bbZ}$, which is a CAT(0)-group. 
Let $K =\bbZ{\times}{\bbZ}/2{\bbZ} = {\langle}t_1, t_2{\rangle}$, which is a virtually infinite cyclic. 
Also, set $K_2 < G$ with presentation:
$${\langle} x_1, x_2, t_1, \xi :   t_1^2 = 1, t_1x_1t_1^{-1} = x_1^{-1}, t_1x_2t_1^{-1} = x_2,
{\xi}x_1{\xi}^{-1} =x_2^{-k}x_1x_2^{k}, {\xi}x_2{\xi}^{-1} = x_2, [t_1, \xi ] = 1{\rangle}.$$ Notice that $G=K_1*_KK_2$. In order to show that $G$ is CAT$(0)$,
it suffices to show that $K_2$ is a CAT$(0)$ group. 

To that end, we change generators in $K_2$ by setting ${\zeta} = x_2^k{\xi}$. Then the presentation of $K_2$ becomes:
$${\langle} x_1, x_2, t_1, \zeta :   t_1^2 = 1, t_1x_1t_1^{-1} = x_1^{-1}, t_1x_2t_1^{-1} = x_2,
{\zeta}x_1{\zeta}^{-1} = x_1, {\zeta}x_2{\zeta}^{-1} = x_2, [t_1, \zeta ] = 1{\rangle}.$$
For Lemma \ref{lem-aut-F2}, Case 1, $K_2$ is CAT(0) and we are done.
%
%Now, let $L_1 = {\langle} x_2, t_1, \zeta :   t_1^2 = [t_1, x_2] = [t_1, \zeta] = [x_2, \zeta] = 1{\rangle} < K_2$ which is isomorphic to ${\bbZ}^2{\times}{\bbZ}/2{\bbZ}$
%and thus it is CAT(0). Also, let $L_2 = {\langle} x_1, t_1, \zeta :   t_1^2 = 1, t_1x_1t_1^{-1} = x_1^{-1}, 
%{\zeta}x_1{\zeta}^{-1} = x_1, [t_1, \zeta ] = 1{\rangle}$. Then 
%$$L_2 = {\langle}{\zeta}{\rangle}{\times}{\langle} x_1, t_1:   t_1^2 = 1, t_1x_1t_1^{-1} = x_1^{-1}{\rangle} \cong {\bbZ}{\times}D_{\infty}.$$
%The infinite dihedral group is CAT(0)-group because it is a Coxeter group (\cite{mou}). Thus $L_2$ is a CAT(0)-group, as a product of CAT(0)-group. Also,
%$L = {\langle}t_1, {\zeta} :\; t_1 ^2 = [t_1, {\zeta} = 1{\rangle} \cong {\bbZ}{\times}{\bbZ}/2{\bbZ}$. 
%But $K_2 = L_1*_LL_2$ is CAT(0) (\cite{brha}, Part II, Corollary 11.19). Similarly $F_3{\rtimes}(\bbZ{\times}{\bbZ}/2{\bbZ}) = K_1*_KK_2$ is a CAT(0)-group.

\vspace{12pt}
\noindent
\underline{Case 2}. We assume $t_2 \in \aut(F_2)$. 
%Let $t_2(x_1) = w_1(x_1, x_2)$ and $t_2(x_2) = w_2(x_1, x_2)$. Since $t_1t_2 = t_2t_1$, we have that
%$$w_1(x_1^{-1}, x_2) = w_1(x_1, x_2)^{-1}, \;\; w_2(x_1^{-1}, x_2) = w_2(x_1, x_2).$$
%The second relation implies that the word $w_2$ does not contain any $x_1$. We can assume that $w_2 = x_2$. The first relation implies that $w_1 = cx_1c^{-1}$. But $w_1$ and $x_2$ must be a generating set for $F_2$. That mean that $w_1 = x_2^k$, $k \in \bbZ$. 
Using Lemma \ref{lem-aut-F2}, Part (2),  the group $G$ has
presentation:
$$\begin{array}{ll}
{\langle} x_1, x_2, \ldots, x_n, t_1, t_2: &  t_1^2 = 1, t_1x_1t_1^{-1} = x_1^{-1}, t_1x_it_1^{-1} = x_i, i=2,\ldots, x_n, \\
&t_2x_1t_2^{-1} = x_2^kx_1^{{\pm}1}x_2^{-k}, t_2x_2t_2^{-1} = x_2^{{\pm}1}, t_2x_it_2^{-1} = x_i, i=3,\ldots,n,  [t_1, t_2] = 1{\rangle}
\end{array}$$
Set $K_1 = {\langle}t_1, t_2, x_3,\ldots, x_n:\; t_1^2 = [t_1, t_2] = [t_1, x_i] = [t_2, x_i] = 1, i=3,\ldots,n {\rangle} \cong {\bbZ}^{n-1}{\times}{\bbZ}/2{\bbZ}$ and
$K = {\langle}t_1, t_2{\rangle} \cong {\bbZ}{\times}{\bbZ}/2{\bbZ}$ which are two subgroups of $G$. Finally, set $K_2$ to be
$${\langle} t_1, t_2, x_1, x_2:\;  t_1^2 = 1, t_1x_1t_1^{-1} = x_1^{-1}, t_1x_2t_1^{-1} = x_2, 
t_2x_1t_2^{-1} = x_2^kx_1^{{\pm}1}x_2^{-k}, t_2x_2t_2^{-1} = x_2^{{\pm}1}, [t_1, t_2] = 1{\rangle}.$$ It is obvious that $G=K_1*_KK_2$. To show that $G$ is CAT$(0)$
it suffices to show that $K_2$ is a CAT$(0)$ groups. 
 
Set ${\zeta} = x_2^{-k}t_2$ and the presentation becomes:
$${\langle} t_1, x_1, x_2, \zeta :\;  t_1^2 = 1, t_1x_1t_1^{-1} = x_1^{-1}, t_1x_2t_1^{-1} = x_2, 
{\zeta}x_1{\zeta}^{-1} = x_1^{{\pm}1}, {\zeta}x_2{\zeta}^{-1} = x_2^{{\pm}1}, [t_1, {\zeta}] = 1{\rangle},$$
which is CAT(0) from Lemma \ref{lem-aut-F2}.
%We set $L_1 = {\langle}{\zeta}, t_1, x_2: t_1^2 = [t_1, x_2] = [t_1, {\zeta}] = 1, {\zeta}x_2{\zeta}^{-1} = x_2^{{\pm}1}{\rangle} \cong {\bbZ}^2{\times}{\bbZ}/2{\bbZ}$ and 
%$$L_2 =
%{\langle}t_1, x_1, {\zeta}: \; t_1^2 = 1, t_1x_1t_1^{-1} = x_1^{-1}, [t_1, {\zeta}] = [{\zeta}, x_1]= 1{\rangle} = {\langle}{\zeta}{\rangle}{\times}
%{\langle}t_1, x_1: \; t_1^2 = 1, t_1x_1t_1^{-1} = x_1^{-1}{\rangle} $$
%which is isomorphic to ${\bbZ}{\times}D_{\infty}$. The proof is completed as in the first case.
\end{proof}

%\begin{lem}
%The group $G_n=F_n\rtimes(F_{n-1}(\rtimes(\ldots(F_3\rtimes(\bbZ\times\bbZ/2\bbZ)\ldots)<
%\h_{(n-1)}$ is a CAT$(0)$-groups.
%\end{lem}
%
%\begin{proof}
%The proof is inductive. The first step is shown in the previous Lemma for $n=3$. To prove the 
%inductive step, we see that a presentation of $G_n$ is of the form
%$$\< G_{n-1}, x_n : \; t_1x_nt_1^{-1}=x_n, t_2x_nt_2^{-1}=x_i^kx_nx_i^{-k}\>$$
%for some $i\in\{ 2,\ldots, n-1\}$, if $t_2$ comes from $\hol{(G_{n-1})}$ and
%$$\<G_{n-1},x_n :\; t_1x_nt_1^{-1}=x_n, t_2x_nt_2^{-1}=x_n\>$$ if $t_2$
%comes from $\aut(G_{n-1})$.   The second case is isomorphic to $G_{n-1}\times\bbZ$
%and so is CAT$(0)$ since $G_{n-1}$ is CAT$(0)$ by inductive hypothesis.
%
%In the first case, we change generators again by setting $\xi=x_i^{-k}t_2$ and rewrite
%the presentation to $\< G_{n-1},x_n:\; t_1x_nt_1^{-1}=x_n, \xi x_n\xi^{-1}=x_n\>$
%and the result follows as before.
%\end{proof}

The reader should notice that there are more possibilities for the $\bbZ/2\bbZ$ action on the $F_2\rtimes\bbZ$ subgroups
of $\aut(F_2)$.

\end{document}